\documentclass[10pt,oneside,reqno,final]{amsart}

\usepackage{hyperref}
  \hypersetup{
    colorlinks=true,
    linktocpage=true,
    linkcolor=blue,
    citecolor=blue,
    urlcolor=blue,
    pdftitle={On the multipliers at fixed points of self-maps},
    pdfauthor={Adolfo Guillot and Valente Ramirez}
  }

\usepackage{stmaryrd}
\usepackage{amssymb}

\newtheorem{theorem}{Theorem}[section]

\newtheorem{fact}[theorem]{Fact}

\theoremstyle{definition}

\newtheorem{problem}[theorem]{Problem}

\theoremstyle{remark}
\newtheorem{remark}[theorem]{Remark}

\numberwithin{equation}{section}

\usepackage[T1]{fontenc}
\usepackage[utf8]{inputenc}

\newcommand{\indel}[1]{\partial/\partial #1}
\newcommand{\trace}{\mathrm{tr}}
 
\begin{document}

\title[On the multipliers at fixed points of self-maps]{On the multipliers at fixed points\\of self-maps of the projective plane}

\author{Adolfo Guillot}
\address{Instituto de Matematicas, Universidad Nacional Aut\'onoma de M\'exico, Ciudad Universitaria, Ciudad de M\'exico 04510, Mexico}
 
\email{\href{mailto:adolfo.guillot@im.unam.mx}{adolfo.guillot@im.unam.mx}}
\thanks{A.G.~gratefully acknowledges support from grant PAPIIT-IN102518 (UNAM, Mexico)}

\author{Valente Ram\'{\i}rez}
\address{Institut de Recherche Math\'{e}matique de Rennes\\ Universit\'{e} de Rennes 1\\ UMR 6625\\ Rennes\\ France}
\email{\href{mailto:valente.ramirez@univ-rennes1.fr}{valente.ramirez@univ-rennes1.fr}}
\thanks{V.R.~was supported by the grants PAPIIT IN-106217 and CONACYT 219722. He also acknowledges the support of the Centre Henri Lebesgue ANR-11-LABX-0020-01.}




\begin{abstract}
This paper deals with holomorphic self-maps of the complex projective plane and the algebraic relations among the eigenvalues of the derivatives at the fixed points. These eigenvalues are constrained by certain index theorems such as the holomorphic Lefschetz fixed-point theorem. A simple dimensional argument suggests there must exist even more algebraic relations that the ones currently known. In this work we analyze the case of quadratic self-maps having an invariant line and obtain all such relations. We also prove that a generic quadratic self-map with an invariant line is completely determined, up to linear equivalence, by the collection of these eigenvalues. Under the natural correspondence between quadratic rational maps of~$\mathbb{P}^2$ and quadratic homogeneous vector fields on~$\mathbb{C}^3$, the algebraic relations among multipliers translate to algebraic relations among the Kowalevski exponents of a vector field. As an application of our results, we describe the sets of integers that  appear as the Kowalevski exponents of a class of quadratic homogeneous vector fields on~$\mathbb{C}^3$ having exclusively single-valued solutions.
\end{abstract}

\maketitle

\section{Introduction}

Let~$\mathbb{P}^n$ be the complex projective space of dimension~$n$, $f:\mathbb{P}^n\to \mathbb{P}^n$ a holomorphic map. For each fixed point~$p$ of~$f$, we may consider the \emph{multipliers} of~$f$ at~$p$, i.e.~the eigenvalues of the derivative of~$f$ at~$p$. The set of multipliers of a self-map is intrinsically associated to it, in the sense that two self-maps which are linearly equivalent (conjugate by an automorphism of~$\mathbb{P}^n$) have the same sets of multipliers.

A classical result relating the multipliers at the fixed points of a self-map is the \emph{holomorphic Lefschetz fixed-point theorem} \cite[chapter~3, section~4]{GH}. This theorem, in the particular case of projective spaces, implies that if in~$\mathbb{P}^n\times\mathbb{P}^n$ the graph of~$f$ is transverse to the diagonal, then the multipliers satisfy the following formula:
\begin{equation}\label{eq:holomorphic-Lefschetz}
  \sum_{f(p)=p} \frac{1}{\det(\mathbf{I}-Df|_p)} = 1.
\end{equation}
For~$n=1$, this relation was announced by Fatou in 1917~\cite{fatou}, although Julia claimed priority over it~\cite{julia} (an interesting historical account of these events appears in~\cite[section~III.2]{audin}).

Any holomorphic self-map~$f\colon\mathbb{P}^n\to\mathbb{P}^n$ is given in homogeneous coordinates by \[[z_1:\cdots:z_{n+1}] \mapsto [P_1(z_1,\ldots,z_{n+1}):\cdots:P_{n+1}(z_1,\ldots,z_{n+1})],\] for some homogeneous polynomials~$P_i$ of the  same  degree~$d$,  which is called the \emph{(algebraic) degree} of the map. 
The following question arises naturally:

\begin{problem}\label{prob:missing-relations}
Given a degree~$d$ and a dimension~$n$, find all the algebraic relations among the the eigenvalues of the derivatives at the fixed points of generic mappings~$f:\mathbb{P}^n\to\mathbb{P}^n$ of degree~$d$.
\end{problem}

For rational maps of~$\mathbb{P}^1$, relation~(\ref{eq:holomorphic-Lefschetz}) is the only one, as remarked early on by Fatou~\cite[p.~168]{fatou-memoire} (see also \cite[problem~12-d]{milnor-one} and  the first part of the proof of theorem~1 in~\cite{fujimura}).

In the next simplest instance, this problem concerns self-maps~$f:\mathbb{P}^2\to\mathbb{P}^2$ given in homogeneous coordinates by three homogeneous quadratic polynomials. Such a self-map depends upon $17=6\times3-1$ parameters (a homogeneous quadratic polynomial in three variables depends upon six parameters, and we need three of them; multiplying all three by the same constant yields the same map). A generic self-map has seven fixed points, from which we extract~$14=7\times 2$ eigenvalues. The group~$\mathrm{PGL}(3,\mathbb{C})$, which has dimension~$8$, acts upon the space of self-maps, and two linearly equivalent ones have the same multipliers. Thus, the~$14$ eigenvalues depend essentially upon~$9=17-8$ parameters: there must be at least five algebraically independent relations among them.

Three of these relations are well known. A quadratic map has, generically, seven fixed points~$p_1,\ldots, p_7$. Let~$u_i$ and~$v_i$ denote the eigenvalues of~$\mathbf{I}-Df|_{p_i}$. We have the relations
\begin{equation}\label{knowneq0} \sum_{i=1}^7\frac{1}{u_iv_i}   =1, \end{equation}
\begin{equation}\label{knowneq1} \sum_{i=1}^7\frac{u_i+v_i}{u_iv_i}   =4, \end{equation} 
\begin{equation}\label{knowneq2} \sum_{i=1}^7\frac{(u_i+v_i)^2}{u_iv_i}  =16. \end{equation}

The first one comes directly from the holomorphic Lefschetz fixed-point theorem~\cite[chapter~3, section~4]{GH}. The others may be obtained by a suitable direct application \cite[section~2]{ramirez2016woods} of the Atiyah-Bott fixed-point theorem~\cite[section~4]{ab-ii}, a generalization of Lefschetz's one.  We remark that these formulas have been constructed by several authors using different techniques cf.~\cite{Ueda1994}, \cite{guillot-pointfixe}, \cite{Abate2014}.

What are the missing relations? This question turns out to be more elusive than it seems at first. Such relations need not come from index theorems; and while the known relations have brief and simple expressions, the missing ones may have extremely complicated ones. 
We do not know the answer of this question even in the case of self-maps of~$\mathbb{P}^2$ given by quadratic polynomials, but we are able to give some answers for those that have an invariant line.
 
\subsection{Quadratic self-maps with an invariant line}

For a self-map $f:\mathbb{P}^2\to \mathbb{P}^2$ and a line~$\ell\subset\mathbb{P}^2$, we say that~$\ell$ is  \emph{(forward) invariant} or that~$f$ \emph{preserves}~$\ell$ if~$f(\ell)\subset\ell$. If~$f$ is of degree two and has an invariant line~$\ell$,  three of the seven fixed points, say, $p_5,p_6$ and~$p_7$, are on~$\ell$. If the corresponding eigenvalues of~$\mathbf{I}-Df$ are~$u_i$ and~$v_i$, with~$u_i$ associated to the eigenvalue of~$f|_\ell$ at~$p_i$, we have two more relations among the multipliers:
\begin{equation}\label{rel-invline1} \sum_{i=5}^7\frac{1}{u_i}=1,  \end{equation}
\begin{equation}\label{rel-CS}\sum_{i=5}^7\frac{v_i}{u_i}=1.
\end{equation}
These relations come, respectively, from the Lefschetz fixed-point theorem applied to~$f|_\ell$ and  from another suitable application of the Atiyah-Bott fixed-point theorem \cite[section~4]{ramirez2016woods} (see also appendix~\ref{appendix:relative-formula} for an elementary proof). 

In the variety of quadratic self-maps of~$\mathbb{P}^2$, those which have an invariant line form a codimension one subvariety. In  this variety, since the fourteen eigenvalues depend upon eight parameters, there are at least six independent relations. We know explicitly five of them so  we are  still missing  at least one.

Let us denote by $t_i, d_i$ the trace and determinant of $\mathbf{I}-Df|_{p_i}$ for $i=1,\ldots,4$, that is, $t_i=u_i+v_i$, $d_i=u_iv_i$.

The following theorem solves problem \ref{prob:missing-relations} for the case of quadratic maps with an invariant line.

\begin{theorem}\label{thm:new-relation}
Let~$B=\mathbb{C}(t_1,d_1,\ldots,t_4,d_4)$ be the field of rational functions on~$\mathbb{C}^8$ and let~$F=\mathbb{C}(u_5,v_5,\ldots,u_7,v_7)^{S_3}$ be the field of rational functions on~$(\mathbb{C}^2)^3$ that are invariant under the action of~$S_3$. 
\begin{itemize}
\item For every~$\sigma\in F$ there exists a polynomial $h_\sigma\in B[x]$ of degree at most four such that any quadratic map $f$ with an invariant line and non-degenerate fixed points satisfies
\begin{equation}\label{eq:new-relation}
  h_\sigma|_{(t_1,d_1,\ldots,t_4,d_4)}(\sigma|_{(u_5,v_5,\ldots,u_7,v_7)})\equiv 0.
\end{equation}
\item For each one of the following elements of~$F$,
\begin{align*}
e_{1,0} &= u_5+u_6+u_7, \\
e_{3,0} &= u_5u_6u_7, \\
e_{0,1} &= v_5+v_6+v_7,
 \\ e_{1,1} &= u_5v_6+u_5v_7+u_6v_7+u_6v_5+u_7v_5+u_7v_6,
\end{align*}
the corresponding relation is algebraically independent of the previously known relations~(\ref{knowneq0})--(\ref{rel-CS}). 
\end{itemize}
\end{theorem}

We shall see in section~\ref{sec:generating-relations} that the relations associated to these four polynomials, together with the previously known ones, generate all the relations among the multipliers (in a sense to be made precise by theorem~\ref{thm:generate}).

Unfortunately, we are unable to explicit any one of the formulas that follow from this theorem. They are very likely to be extremely intricate, even more than those which appear in the related problem studied in~\cite{Kudryashov-R}.
 
Explicit formulas may be used to determine if a given set of seven couples of numbers is the set of multipliers of a quadratic self-map of~$\mathbb{P}^2$ with an invariant line.  For the problems related to differential equations which will be discussed in section~\ref{sec:pain}, this application is a very relevant one. For it, we can overcome the absence of explicit formulas by having a computer program that tests whether a set of potential multipliers  is realizable by a  self-map or not (or, more precisely, that tests whether it is in the Zariski closure of those which are).
We have programmed such a test, which from now on we will refer to as~Test\footnote{The Test may be found in the ancillary files as~\texttt{Test.sage}.}, and used it for the applications of section~\ref{sec:pain}; we will give some theoretical and practical  details for it in section~\ref{sec:test}.

There is another natural question regarding the multipliers of fixed points.

\begin{problem}\label{prob:find-N}
To what extent do the multipliers of a self-map determine the~map?
\end{problem} 
 
For maps of~$\mathbb{P}^1$, in the quadratic case the spectra completely determines the map up to a linear change of coordinates~\cite[lemma~3.1]{milnor-quad} (for higher degrees, non-equivalent maps with the same spectra form positive-dimensional varieties).
On a closely related subject, we have that for polynomial maps of~$\mathbb{C}$ of degree~$d$, there are, generically,  $(d-2)!$ affine equivalence classes of maps realizing a (realizable) set of multipliers (see~\cite[appendix~A]{milnor-cubic} for~$d=3$, \cite{fujimura} for~$d>3)$.

For quadratic self-maps of~$\mathbb{P}^2$, by proposition~3.2 in~\cite{guillot-semi} (whose  incomplete proof is completed in~\cite{guillot-qc3}), there is a natural number~$N$ such that, generically, there are~$N$ linear equivalence classes of self-maps having the same multipliers. This implies that, modulo linear equivalence, the multipliers completely determine a  self-map up to some finite ambiguity. The above problem requires us to determine this number~$N$. From the main result in~\cite{linsneto240} and the proof of lemma~5 in~\cite{guillot-pointfixe}, this~$N$ is at most~$240$.

In the case of quadratic maps with an invariant line we can show that this number~$N$ is actually equal to one.

\begin{theorem}\label{thm:1point}
A generic self-map having an invariant line is completely determined, up to linear equivalence, by the set of multipliers of all of its fixed points.
\end{theorem}

For the quadratic maps of~$\mathbb{P}^2$ that come from polynomial maps of~$\mathbb{C}^2$ (those with a line which is both forward and backward invariant), the results in~\cite{Kudryashov-R} and~\cite{ramirez-twin}, when suitably reinterpreted (see remark~\ref{rem:totinv}), give a complete answer to problems~\ref{prob:missing-relations} and~\ref{prob:find-N}. Our results extend them to the case where there is a forward invariant line  which is not necessarily backward invariant.

\subsection{An application to the Painlev\'e analysis of some differential equations} As an application of our results, and in order to exhibit a situation where the possibility of doing the effective computations enabled by our~Test is critical, we will study some ordinary differential equations in the complex domain given by vector fields of the form
\begin{equation}\label{vf:quadint}\sum_{i=1}^3P_i(z_1,z_2,z_3)\frac{\partial}{\partial z_i},\end{equation} 
with~$P_i$ a quadratic homogeneous polynomial. We will consider the problem of understanding those that, as ordinary differential equations, do not have multivalued solutions. To a vector field~$X$ like~(\ref{vf:quadint}) we may associate the self-map~$f$ of~$\mathbb{P}^2$ given by~$[z_1:z_2:z_3]\mapsto [P_1:P_2:P_3]$. For a non-degenerate fixed point~$p$ of the latter, the eigenvalues of~$\mathbf{I}-Df|_{p}$ are said to be \emph{Kowalevski exponents} of~$X$. These exponents give an obstruction of arithmetical nature for the absence of multivalued solutions: \emph{if the differential equation does not have multivalued solutions, all its Kowalevski exponents are integers}. Our results open the possibility of classifying the quadratic homogeneous equations in three variables which have an invariant plane and which do not have multivalued solutions. We will study a particular instance of this problem in section~\ref{sec:pain}.

We have made intensive use of  symbolic computations in the proofs of our results.  The main object of this article, the map that to a quadratic map of~$\mathbb{P}^2$ with an invariant line  associates the multipliers at its fixed points, has been analyzed with the help of the computer.
We also have  the previously described program that tests if a set of potential multipliers  is realizable by a  self-map, and the calculations of a combinatorial number-theoretic nature that occur in section~\ref{sec:pain}. We have used the computer algebra system SageMath~\cite{sagemath} for most of these implementations. Additional Gr\"{o}bner basis computations have been made using Macaulay2~\cite{Macaulay2}. All of our code is available at this article's Github repository \cite{GitHubRepo}, and as ancillary files in its arXiv's abstract page.

The authors thank Jawad Snoussi and Javier Elizondo for helpful conversations, and Marco Abate and Masayo Fujimura for providing  useful references.

\section{Proofs}

We will now give proofs of the results stated in the introduction. For this, we will need to make some explicit calculations related to a particular  map. The results of these calculations appear in section~\ref{sec:case-study} and will be cited all along this section as they are being required.
 
\subsection{Introduction} Any holomorphic self-map~$f\colon\mathbb{P}^2\to\mathbb{P}^2$ is defined, in homogeneous coordinates, by a correspondence of the form~$[z_1:z_2:z_3] \mapsto [P_1:P_2:P_3]$, where $P_1$, $P_2$ and~$P_3$ are homogeneous polynomials of the same degree~$d$ and without common zeros (the last condition excludes the existence of indeterminacy points.) The polynomials $P_i$ are uniquely defined up to rescaling by a common non-zero complex number. The number~$d$ is called the (algebraic) \emph{degree} of~$f$. It follows from the topological Lefschetz fixed point formula that a holomorphic self-map of degree $d$ has exactly $d^2+d+1$ fixed points, counted with multiplicity. 

A fixed point~$p$ of a  self-map~$f$ is said to be \emph{non-degenerate} if it has multiplicity one as a fixed point or, equivalently, if  the graph of~$f$ in $\mathbb{P}^2\times\mathbb{P}^2$ intersects the diagonal transversely at $p$, or, yet, if~$\det(\mathbf{I}-Df|_{p}) \neq 0$. 
 
Fix a line $\ell\subset\mathbb{P}^2$. A holomorphic quadratic self-map of~$\mathbb{P}^2$ preserving~$\ell$ and having only isolated and simple fixed points has exactly three fixed points on $\ell$ and four fixed points away from $\ell$. Let~$E^{(4)}$ be the variety of holomorphic quadratic self-maps of~$\mathbb{P}^2$  preserving~$\ell$ and with isolated and simple fixed points, and having the four fixed points away from~$\ell$ marked and ordered: $p_1$, $p_2$, $p_3$, $p_4$. We denote the remaining fixed points by $p_5$, $p_6$, $p_7$ (the numbering is not unambiguously defined throughout $E^{(4)}$). Let us denote by~$\operatorname{Aut}(\mathbb{P}^2,\ell)$ the subgroup of the group~$\mathrm{PGL}(3,\mathbb{C})$ of automorphisms of~$\mathbb{P}^2$ that preserve~$\ell$; in the affine chart~$\mathbb{P}^2\setminus\ell$, these are precisely the affine transformations.

\subsection{The spectral map \texorpdfstring{$\Psi$}{Psi}}\label{sec:mapPsi}

The first step in the analysis of the spectra of a quadratic self-map with invariant line is the following theorem, which is analogous to the main result of \cite{ramirez-twin} for quadratic vector fields on $\mathbb{C}^2$ (see theorem~\ref{thm:twins} below). 

Recall we have defined $t_i, d_i$ to be the trace and determinant of $\mathbf{I}-Df|_{p_i}$, for $i=1,\ldots,4$.

\begin{theorem}\label{thm:4points}
The algebraic map $\Psi:E^{(4)}\to \mathbb{C}^8$ given by~$f\mapsto(t_1,d_1,\ldots,t_4,d_4)$ is dominant and its generic fiber consists of exactly four orbits of the natural action of $\operatorname{Aut}(\mathbb{P}^2,\ell)$ on $E^{(4)}$.
\end{theorem}

Linearly equivalent self-maps have the same image under $\Psi$, so the map factors through to the (Geometric Invariant Theory) quotient:
\[
  \Psi\colon E^{(4)}\longrightarrow E^{(4)}\sslash\operatorname{Aut}(\mathbb{P}^2,\ell) \longrightarrow \mathbb{C}^8.
\]
Both~$E^{(4)}\sslash\operatorname{Aut}(\mathbb{P}^2,\ell)$ and~$\mathbb{C}^8$ have dimension eight, so one may ask if the generic fiber of the induced map between them is finite. This is indeed the case, and the proof is immediate: we need only consider a generic element $f\in E^{(4)}$ and compute the rank of the derivative $D\Psi(f)$. We have done this for the map $f$ discussed in section~\ref{sec:case-study}. The rank is maximal which implies that the map~$\Psi$ is dominant and, after taking the quotient, the generic fiber is finite. Thus, there exists a natural number~$m$ such that the generic fiber of $\Psi$ is the union of exactly~$m$ orbits of the action of $\operatorname{Aut}(\mathbb{P}^2,\ell)$. In the next section we will show that~$m=4$, thus proving the above theorem.

\subsection{Auxiliary vector field and the Jacobi system}

Let us assume now without loss of generality that, with respect to some homogeneous coordinates~$[z_1:z_2:z_3]$, the line $\ell$ is given by the equation $z_1=0$. A quadratic self-map~$f$ of~$\mathbb{P}^2$ leaving this line invariant may be written as
\[ [z_1:z_2:z_3]\mapsto [z_1\widehat{L}:z_2\widehat{L}-\widehat{P}:z_3\widehat{L}-\widehat{Q}], \]
for some linear homogeneous form~$\widehat{L}$ and two quadratic homogeneous polynomials~$\widehat{P}$, $\widehat{Q}$. A point~$p\in\mathbb{P}^2$ which is not on the line~$z_1=0$ is fixed by~$f$ if~$\widehat{P}(p)=0$ and~$\widehat{Q}(p)=0$ but~$\widehat{L}(p)\neq 0$ (for otherwise the point will be an indeterminacy one).

In the affine chart~$z_1\neq 0$ with affine coordinates~$[1:x:y]$, $f$ reads
\begin{equation}\label{for:mapinaffchart}(x,y)\mapsto \left(x-\frac{P(x,y)}{L(x,y)},y-\frac{Q(x,y)}{L(x,y)}\right),\end{equation}
for~$P(x,y)=\widehat{P}(1,x,y)$, $Q(x,y)=\widehat{Q}(1,x,y)$ and~$L(x,y)=\widehat{L}(1,x,y)$. The fixed points of~$f$  within this affine  chart are the common solutions of~$P=0$ and $Q=0$. We have
\begin{multline}\label{cta} \mathbf{I}-Df|_{p_i}=\left.\frac{1}{L(p_i)^2}\left(\begin{array}{cc} LP_x-PL_x & LP_y-PL_y \\  LQ_x-QL_x & LQ_y-QL_y
\end{array}\right)\right|_{p_i} = \\ = \left.\frac{1}{L(p_i)}\left(\begin{array}{cc}   P_x  & P_y  \\  Q_x  & Q_y 
\end{array}\right)\right|_{p_i},\end{multline}
where, for the last equality, we have used~$P(p_i)=0$ and~$Q(p_i)=0$ ($L$ does not vanish at these points).

To the self-map~$f$ we may associate the auxiliary quadratic (inhomogeneous) vector field on~$\mathbb{C}^2$ 
\[ H=P\frac{\partial}{\partial x}+Q\frac{\partial}{\partial y}.\]
The correspondence is an equivariant one. The singular points of~$H$ correspond to the fixed points of~$f$, and their spectral data are intimately related: the last matrix in~(\ref{cta}) is the linear part of~$H$ at~$p_i$ and thus
\begin{equation}\label{endo-campo-base}
DH|_{p_i} = L(p_i) (\mathbf{I}-Df|_{p_i}).
\end{equation}
The couple~$(L,H)$ determines the self-map~$f$. In order to prove theorem~\ref{thm:4points} we will prove that, given some generic~$(t_1,d_1,\ldots,t_4,d_4)$, there are, up to equivalence under~$\operatorname{Aut}(\mathbb{P}^2,\ell)$, exactly four couples~$(L,H)$ giving self-maps with the prescribed data. We will actually see that there are two possible choices for~$L$ and that, for each one of these choices, there are two choices for~$H$.

\begin{remark} \label{rem:totinv}
The case where~$\ell$ is both forward and backward invariant corresponds to the case where~$L\equiv 1$, which makes the map~(\ref{for:mapinaffchart}) polynomial. Reciprocally,  a quadratic polynomial map of~$\mathbb{C}^2$ may be seen as a map of~$\mathbb{P}^2$ with a totally invariant line.
In this case, the map $f$ is completely determined by the auxiliary vector field~$H$. Moreover, for the multipliers of  the fixed points that lie on the invariant line, a straightforward calculation shows that~$v_i=1$ for all~$i$; and so, all the information is concentrated in the~$u_i$.
The main result of~\cite{Kudryashov-R}, which is originally given in terms of quadratic vector fields but which can be easily adapted to our context, gives an explicit quadratic polynomial~$h$ with coefficients in~$\mathbb{C}[t_1,d_1,\ldots, t_4,d_4]$ such that every quadratic self-map with a forward and backward invariant line satisfies $h\vert_{(t_1,d_1,\ldots, t_4,d_4)}(u_5u_6u_7) \equiv 0$.
\end{remark}

Jacobi's formula, which is a direct application of the residue theorem in dimension two \cite[chapter~5, section~2]{GH}, states in our situation that if~$g(x,y)$ is a polynomial of degree at most~$1$,
\[\sum_{i=1}^4\frac{g(p_i)}{\det(DH|_{p_i})}=0.\]
Taking for~$g$ the polynomials~$1$, $x$, $y$, $\trace(DH)$ and~$L$, and taking~(\ref{endo-campo-base}) into account, we obtain, respectively, the relations
\begin{align}
\sum_{i=1}^4\frac{1}{\det(DH|_{p_i})} & =\sum_{i=1}^4\frac{1}{L^2(p_i)d_i}=0, \label{jac1}\\
\sum_{i=1}^4\frac{x(p_i)}{\det(DH|_{p_i})} & =\sum_{i=1}^4\frac{x(p_i)}{L^2(p_i)d_i}=0,\label{jac2}\\
\sum_{i=1}^4\frac{y(p_i)}{\det(DH|_{p_i})} & =\sum_{i=1}^4\frac{y(p_i)}{L^{2}(p_i)d_i}=0, \label{jac3}\\
\sum_{i=1}^4\frac{\trace(DH|_{p_i})}{\det(DH|_{p_i})} &  =\sum_{i=1}^4\frac{t_i}{L(p_i)d_i}=0 \label{jac4}\\ 
\sum_{i=1}^4\frac{L(p_i)}{\det(DH|_{p_i})} & =\sum_{i=1}^4\frac{1}{L(p_i)d_i}=0.  \label{jac5}
\end{align}

\begin{proof}[Proof of theorem \ref{thm:4points}]
First of all, note that the points $p_1,p_2,p_3$ are not collinear (otherwise the line through them would  also be invariant and would therefore have a fourth fixed point at its intersection with~$\ell$, and this would imply that the line is pointwise fixed, in contradiction to our hypotheses). Therefore, we may suppose, up to a transformation~$g\in\operatorname{Aut}(\mathbb{P}^2,\ell)$, that~$p_1=(0,0)$, $p_2=(1,0)$, and~$p_3=(0,1)$. Since~$g$ exists and is unique, this amounts to factoring out the action of~$\operatorname{Aut}(\mathbb{P}^2,\ell)$ on~$E^{(4)}$. We will write~$L=ax+by+c$ and~$p_4=(x_4,y_4)$. Since~$c=L(p_4)$, $c\neq 0$. By multiplying $L$, $P$ and $Q$ by a common scalar factor we may assume without loss of generality that $c=1$.
Let us also recall that the non-degeneracy of the fixed points implies that $d_i\neq0$.
 
By considering the numbers $t_i,d_i$ as parameters and after lifting denominators, the above equations give a system of five polynomial equations in the four unknowns~$a$, $b$, $x_4$, $y_4$. 
From these equations we can deduce that $b$ satisfies a quadratic equation
\begin{equation}\label{eq:quadratic-b}
 s_2b^2 + s_1b + s_0 = 0,
\end{equation}
where $s_j$ are polynomials on $t_i,d_i$. Moreover, each of the remaining variables $a$, $x_4$, $y_4$, satisfies a linear equation with polynomial coefficients depending on $t_i,d_i$ and $b$.
These claims can be verified\footnote{This is done in the ancillary file \texttt{analysis-jacobi-ideal.m2}.} by computing a Gr\"{o}bner basis in the ring $\mathbb{C}[t_1,\ldots,d_4,x_4,y_4,a,b]$ for the ideal $J$ generated by the five polynomial equations obtained from (\ref{jac1})--(\ref{jac5}).

Let us fix one of the two possible values of~$b$, given as one of the roots of  (\ref{eq:quadratic-b}). This determines~$a$, $x_4$ and~$y_4$. In particular, each one of the values of~$b$ determines a value for~$L$. Let us now show that each determination of~$L$ gives two determinations of~$H$. 
A potential auxiliary vector field~$H$ ought to have equilibrium points at~$p_i$ and the eigenvalues of its linear parts at these points are dictated by the formulas~(\ref{endo-campo-base}). Note that once the spectra of $H$ is fixed, the identities (\ref{jac2}) and (\ref{jac3}) completely determine the position of $p_4$, and so the only real constrain imposed on $H$ is the prescribed spectra. In this way, we have reduced our problem to the problem of realizing quadratic vector fields on $\mathbb{C}^2$ with prescribed spectra. 

Denote by $V^{(4)}$ the space of quadratic vector fields on~$\mathbb{C}^2$ with marked singular points $p_1,\ldots,p_4$. For a vector field $H\in V^{(4)}$, let $(\tau_i,\delta_i)$ be the trace and determinant of $DH(p_i)$. We have a rational map $\psi\colon H\mapsto (\tau_1,\delta_1, \ldots, \tau_4,\delta_4)$. It takes values in the variety $X\subset\mathbb{C}^8$ given by the Jacobi relations (\ref{jac1}) and (\ref{jac4}). We have the following result.

\begin{theorem}[{\cite[theorem~1]{ramirez-twin}}] \label{thm:twins}
The rational map $\psi\colon V^{(4)}\dashrightarrow X$ is dominant and its generic fiber consists of exactly two orbits of the natural action of the group $\operatorname{Aff}(2,\mathbb{C})$ on $V^{(4)}$.
\end{theorem}

Summarizing, we have done the following: starting from the values $t_1,\ldots,d_4$, we look for those $f\in E^{(4)}$ that realize these numbers. We have used the system of equations (\ref{jac1})--(\ref{jac5}) to establish that there are only two possible values that $(a,b)$ may take, and thus only two options for the polynomial $L$. 
For each of these two choices of $L$, the position of $p_4$ is completely determined, as well as the spectra of the auxiliary vector field $H$. By theorem \ref{thm:twins} there exist exactly two vector fields compatible with this data. 
 
Therefore, there exist exactly four self-maps that realize the prescribed multipliers. This establishes theorem~\ref{thm:4points}. \end{proof}

\subsection{Deducing the new relations} 

Let us now prove the first part of theorem~\ref{thm:new-relation}. Fix an element~$\sigma$ in the field $F$ of rational functions on~$(\mathbb{C}^2)^3$ that are invariant under the action of~$S_3$, and define the rational function
\[
  \psi_\sigma \colon E^{(4)} \dashrightarrow \mathbb{C}, \quad f\mapsto\sigma{(u_5,v_5,\ldots,u_7,v_7)}.
\]
This map is well defined precisely because~$\sigma$ is invariant under the action of $S_3$. Now, fix some generic values for the variables $t_1,d_1,\ldots,t_4,d_4$. 
Theorem~\ref{thm:4points} says that there exist, up to linear equivalence, exactly four self-maps~$f_1,\ldots,f_4$ that realize these values. 
Using Vieta's formulas we may construct a monic polynomial $h_\sigma$ whose roots are precisely the four values $\psi_\sigma(f_1),\ldots,\psi_\sigma(f_4)$. 
As we let $t_1,\ldots,d_4$ vary,  the coefficients of $h_\sigma$  are given by algebraic functions with no multivaluation, hence they define rational functions on these variables. By construction, for a generic element $f\in E^{(4)}$ with spectral data $t_1,\ldots,d_4$ and $u_5,\ldots,v_7$, the number $\psi_\sigma(f) = \sigma|_{(u_5,\ldots,v_7)}$ is a root of $h|_{(t_0,\ldots,d_4)}$. 
Therefore, identity (\ref{eq:new-relation}) is satisfied in a Zariski open subset of $E^{(4)}$ and thus in all of $E^{(4)}$. This proves the first part of theorem~\ref{thm:new-relation}.
 
\subsection{Proof of theorem \ref{thm:1point}}

According to theorem~\ref{thm:4points}, generically, the data $(t_1,\ldots,d_4)$ determines finitely many linear equivalence classes of self-maps that realize these values. In the ring of polynomial functions on $(\mathbb{C}^2)^3$ invariant under the action of $S_3$, consider the polynomial~$p=u_1+u_2+u_3$, together with the polynomial~$h_p$ associated to~$p$ via theorem~\ref{thm:new-relation}.  We have calculated~$h_p|_{\tau_0}$ for the particular value~$\tau_0$ of~$(t_1,\ldots,d_4)$ studied in section~\ref{sec:case-study}. The result appears in formula~(\ref{for:test:hp}). This polynomial has four different roots, and thus~$h_p$ has four different roots for almost every choice of $t_1,\ldots,d_4$. 
This implies that the self-map that realizes the extended data $(t_1,\ldots,d_4,p)$ is, up to linear equivalence, unique.

\subsection{Algebraic independence of the new relations} 
 
The first part of theorem~\ref{thm:new-relation} gives some algebraic relations among the multipliers at the fixed points of a rational map. In this section we will prove the second part of theorem~\ref{thm:new-relation}, that affirms that some of these relations are  algebraically independent from the known ones. We will also state and prove theorem~\ref{thm:generate}, which will make precise a sense in which the newly obtained relations, together with the known ones, generate all of the relations.

\subsubsection{Symmetric functions} We begin by recalling some facts about multisymmetric functions. Let~$\mathbb{K}$ be a field. Let~$S_3$ be the group of permutations in three symbols. The polynomials $e_{i,j}$ in~$\mathbb{K}[z_1,w_1,\ldots,z_3,w_3]$ defined by
\[\prod_{k=1}^3 (1+xz_k+yw_k) =\sum_{i,j} e_{i,j}(z_1,w_1,\ldots,z_3,w_3)x^iy^j\]
are invariant under the action of~$S_3$ in~$(\mathbb{K}^2)^3$ and are called \emph{elementary symmetric polynomials} (there are nine nonconstant ones). They generate, although not freely, the ring~$\mathbb{K}[z_1,w_1,z_2,w_2,z_3,w_3]^{S_3}$ of polynomials in~$(\mathbb{K}^2)^3$ that are invariant under the action of~$S_3$. The quotient of~$(\mathbb{K}^2)^3$ under the action of~$S_3$ is rational. Its field of rational functions is freely generated by the functions  
\begin{equation}\label{matt-gen} e_{1,0}, e_{2,0}, e_{3,0}, e_{0,1}, e_{1,1}, e_{2,1}.\end{equation}
This is a particular instance of \emph{Mattuck's theorem}~\cite{mattuck}; see the discussion in~\cite[chapter~4, section~2]{GKZ}. According to Mattuck's theorem, the other elementary symmetric polynomials may be expressed as rational functions of the polynomials~(\ref{matt-gen}).  Explicitly, for the discriminant 
\[\Delta=27e_{3,0}^2-e_{2,0}^2e_{1,0}^2+4e_{2,0}^3+4e_{3,0}e_{1,0}^3-18e_{3,0}e_{2,0}e_{1,0},\]
we have 
\begin{multline}\label{mattuck1}\Delta e_{1,2} = -e_{3,0}e_{2,0}e_{1,1}e_{1,0}e_{0,1}-e_{2,1}e_{2,0}e_{1,1}e_{1,0}^2-3e_{3,0}e_{2,0}e_{1,1}^2+2e_{3,0}e_{2,1}e_{1,0}^2e_{0,1}\\ -3e_{3,0} e_{2,1}e_{1,1}e_{1,0}   -6e_{3,0}e_{2,1}e_{2,0}e_{0,1} +e_{3,0}e_{1,1}^2e_{1,0}^2 +e_{2,1}^2e_{1,0}^3+9e_{3,0}e_{2,1}^2 \\+9e_{3,0}^2e_{1,1}e_{0,1}  +4e_{2,1}e_{2,0}^2e_{1,1}-4e_{2,1}^2e_{2,0}e_{1,0}+e_{3,0}e_{2,0}^2e_{0,1}^2-3e_{3,0}^2e_{1,0}e_{0,1}^2,
\end{multline}
\begin{multline} \Delta e_{0,2}=-e_{2,1}e_{2,0}e_{1,1}e_{1,0}-e_{2,0}^2e_{1,1}e_{1,0}e_{0,1}+ 4e_{3,0}e_{1,1}e_{1,0}^2e_{0,1}  -4e_{3,0}e_{2,0}e_{1,0}e_{0,1}^2 \\-3e_{3,0}e_{2,0}e_{1,1}e_{0,1}-6e_{3,0}e_{2,1}e_{1,0}e_{0,1}  -3e_{2,1}^2e_{2,0}+e_{2,0}^2e_{1,1}^2   +e_{2,0}^3e_{0,1}^2 +e_{2,1}^2e_{1,0}^2 \\   +2e_{2,1}e_{2,0}^2e_{0,1}+9e_{3,0}^2e_{0,1}^2 +9e_{3,0}e_{2,1}e_{1,1}-3e_{3,0}e_{1,1}^2e_{1,0}, \end{multline}
\begin{multline}\label{mattuck3}\Delta e_{0,3}=-e_{2,1}^2e_{1,1}e_{1,0}+e_{2,1}e_{2,0}e_{1,1}^2+e_{2,1}e_{2,0}^2e_{0,1}^2+e_{2,1}^2e_{1,0}^2e_{0,1} -2e_{2,1}^2e_{2,0}e_{0,1}\\  -e_{3,0}e_{2,0}e_{1,1}e_{0,1}^2+e_{3,0}e_{1,1}^2e_{1,0}e_{0,1}  -e_{3,0}e_{1,1}^3  +e_{3,0}^2e_{0,1}^3 -e_{2,1}e_{2,0}e_{1,1}e_{1,0}e_{0,1}\\ +e_{2,1}^3 -2e_{3,0}e_{2,1}e_{1,0}e_{0,1}^2+3e_{3,0}e_{2,1}e_{1,1}e_{0,1}.\end{multline}

\subsubsection{Generating the relations}\label{sec:generating-relations}

Let~$F=(\mathbb{C}^2)^3/S_3$.  Let~$\Phi:E^{(4)}\to F$ be given, for~$f\in E^{(4)}$, by evaluating each one of the elementary symmetric polynomials at the values~$(u_5,v_5,u_6,v_6,u_7,v_7)$ of~$f$. Together with the map~$\Psi$ of section~\ref{sec:mapPsi}, the map~$\Psi\times\Phi:E^{(4)}\to\mathbb{C}^8\times F$ gives all the spectral data. Let~$M\subset \mathbb{C}^8\times F$ be the Zariski closure of the image of~$\Psi\times\Phi$, the \emph{variety of realizable eigenvalues}. The ring  of regular functions on~$\mathbb{C}^8\times F$ contains the ideal~$J_M$ of functions that vanish at~$M$.  We would like to give an explicit list of generators of this ideal, but we will prove a weak version of this.

For an affine variety~$V$, we will denote by~$\mathcal{O}(V)$ its ring of regular (polynomial) functions and by~$\mathcal{M}(V)$ its field of rational ones. We have that~$\mathcal{O}( \mathbb{C}^8\times F)\approx \mathcal{O}( \mathbb{C}^8)\otimes_\mathbb{C}\mathcal{O}(F)$, and we have the natural embedding
\[\mathcal{O}( \mathbb{C}^8)\otimes_\mathbb{C}\mathcal{O}(F)\stackrel{i}{\hookrightarrow}\mathcal{M}( \mathbb{C}^8)\otimes_\mathbb{C}\mathcal{O}(F).\]
We will describe the ideal generated by the  image of~$J_M$ in the ring~$\mathcal{M} (\mathbb{C}^8) \otimes_\mathbb{C} \mathcal{O}(F)$.  The latter is the ring of regular functions of the variety~$(\mathcal{M}(\mathbb{C}^8)^2)^3/S_3$  defined over the field~$\mathcal{M}(\mathbb{C}^8)$. Its points are the rational functions~$r:\mathbb{C}^8\dashrightarrow (\mathbb{C}^2)^3/S_3$. The algebraic extensions of~$\mathcal{M}(\mathbb{C}^8)$ correspond to ramified covers of~$\mathbb{C}^8$.

Consider the following elements of~$\mathcal{M}( \mathbb{C}^8)$:
\[\beta_0   =   1-\sum_{i=1}^4\frac{1}{d_i} ,\quad  \beta_1   = 3 -\sum_{i=1}^4\frac{t_i}{d_i}, \quad \beta_2   = 9-\sum_{i=1}^4\frac{t_i^2}{d_i}.\]
A system of equations equivalent to~(\ref{knowneq0})--(\ref{rel-CS}) is given by
\begin{align}\label{rel:CS} 1-\frac{e_{2,1}}{e_{3,0}} & =   1-\frac{v_5u_6u_7+u_5v_6u_7+u_5u_6v_7}{u_5u_6u_7}=\left(1-\sum_{i=5}^7\frac{v_i}{u_i}\right)=0,  \\ \label{rel:2D} 1-\frac{e_{2,0}}{e_{3,0}} & =    1-\frac{u_5u_6+u_6u_7+u_7u_5}{u_5u_6u_7}=\left(1-\sum_{i=5}^7\frac{1}{u_i}\right)=0, \\
\label{eq:beta2}\beta_2-\frac{e_{1,2}}{e_{0,3}} & =   \beta_2-\sum_{j=5}^7\frac{u_j}{v_j}=\left(16-\sum_{i=1}^7\frac{(u_i+v_i)^2}{u_iv_i}\right)-\left(1-\sum_{j=5}^7\frac{v_j}{u_j} \right)=0, \\
\beta_1-\frac{e_{0,2}}{e_{0,3}} & =  \beta_1-\sum_{j=5}^7\frac{1}{v_j}= \left(4-\sum_{i=1}^7\frac{u_i+v_i}{u_iv_i}\right)-\left(1-\sum_{j=5}^7\frac{1}{u_j} \right)=0,\end{align}
\begin{multline} \label{eq:beta0}\beta_0-\frac{e_{1,1}^2-e_{2,0}e_{0,2}-e_{2,1}e_{0,1}-e_{1,2}e_{1,0}+e_{2,0}e_{0,1}^2+e_{0,2}e_{1,0}^2-e_{1,1}e_{1,0}e_{0,1}}{3e_{0,3}e_{3,0}}= \\ =  \beta_0-\left(\frac{1}{u_5v_5}+\frac{1}{u_6v_6}+\frac{1}{u_7v_7} \right)= \left(1-\sum_{i=1}^7\frac{1}{u_iv_i}\right)=0 .\end{multline}
After clearing denominators (multiplying by~$e_{3,0}$, $e_{0,3}$ or~$e_{3,0}e_{0,3}$), these give elements~$g_1,\ldots,g_5$ of~$\mathcal{M}(\mathbb{C}^8)\otimes_\mathbb{C}\mathcal{O}(F)$.

\begin{theorem}\label{thm:generate} Let~$I$ be the ideal in~$\mathcal{M}( \mathbb{C}^8)\otimes_\mathbb{C} \mathcal{O}(F)$ generated by~$g_1,\ldots,g_5$ and by the polynomials~$h_p$, $h_q$, $h_r$ and~$h_s$ associated by theorem~\ref{thm:new-relation} to 
\begin{equation}\label{for:pqrs} p   =   e_{1,0},\quad  q   =  e_{3,0}, \quad  r  =   e_{0,1},\quad   s  =  e_{1,1}.\end{equation}
Then the ideal generated by~$ i(J_M)$ in $\mathcal{M}(\mathbb{C}^8)\otimes_\mathbb{C} \mathcal{O}(F)$ coincides with $I$. 
\end{theorem}

\begin{proof}
Consider the projection~$\mu:F\to \mathbb{C}^6$ given by the functions in~(\ref{matt-gen}). According to Mattuck's theorem, it is a birational equivalence. By~(\ref{mattuck1})--(\ref{mattuck3}), it is a biholomorphism onto its image in restriction to the complement of the zero locus of~$\Delta$. By composing~$\Phi$ with~$\mu$ we have the map~$\Phi^\flat:E^{(4)}\to \mathbb{C}^6$; it takes values in the 4-plane~$P$ given by equations~(\ref{rel:CS})--(\ref{rel:2D}),
\begin{equation}\label{for:basic} e_{2,1}=e_{3,0},\quad e_{2,0}=e_{3,0}. \end{equation}
Consider the coordinates in~$P$ given by the restrictions of~$p$, $q$, $r$ and~$s$~(\ref{for:pqrs}). We will continue to denote each one of these by the same letter. Substituting~(\ref{mattuck1})--(\ref{mattuck3}) into~(\ref{eq:beta2})--(\ref{eq:beta0}) and then substituting~(\ref{for:basic}), relations~(\ref{eq:beta2})--(\ref{eq:beta0}) become the relations~$R_0$, $R_1$ and~$R_2$ given respectively by
\begin{multline}\label{redrel0}
4sp^2-4p^3r-3qs-12spr-qspr -4s^2p-5qpr^2+\\
qsr -qsp+qp^2 r+10qpr+4sp^2r+ 9s^2+qs^2+\\ q^2r^2 +q^2 -3qp+4p^2r^2+6qr^2-9qr-2q^2r-\beta_0 \Theta=0,\end{multline}
\begin{multline} q^2r^2+qp^2+qs^2-3s^2p+9qs+2q^2r- 3q^2+9qr^2+\\ 4sp^2r-4qpr^2-3qsr-
6qpr-qsp-qspr-\beta_1 \Theta=0,\end{multline}
\begin{multline}\label{redrel2} q^2r^2-qspr-qsp^2-3qs^2+ 4q^2s+s^2p^2-3qsp+ qp^3+\\9qsr+9q^2 -3qp r^2-6q^2r+2qp^2r-4q^2p-\beta_2 \Theta=0,\end{multline}
for  
\begin{multline*}\Theta=q^2+qs^2+q^2r^2+qp^2r-qsr^2 -2qpr^2+\\ 3qsr-s^3+s^2pr-qspr+ qr^3-qsp-2q^2r.\end{multline*}

The map~$\Psi\times\Phi^\flat$ takes values in~$\mathbb{C}^8\times P$. Let~$\pi:\mathbb{C}^8\times P\to \mathbb{C}^8$ be the projection onto the first factor.
We shall study the map~$\Psi\times\Phi^\flat$ in the neighborhood of some generic point. For this purpose we will use the  data analyzed in section~\ref{sec:case-study}.
Let~$\tau_0$ be the point in~$\mathbb{C}^8$ corresponding to the values of~$(t_i,d_i)$ of the map $f$ (given by~(\ref{values-case-study})) map~$f$ in~(\ref{map-case-study}), given by~(\ref{values-case-study}), and let~$P_0=\pi^{-1}(\tau_0)$. Each one of the four polynomials~$h_p,\ldots,h_s$ associated to  $\tau_0$, whose explicit expression is given in (\ref{for:test:hp})--(\ref{for:test:hs}), has a nonvanishing discriminant. These four polynomials vanish simultaneously at~$4^4=256$ different points in~$P_0$, with multiplicity one at each of them. We may calculate how many of these~256 points satisfy relations~(\ref{redrel0})--(\ref{redrel2}) for the values 
\[\beta_0 = \frac{59}{15}, \qquad
  \beta_1 = -\frac{49}{15}, \qquad
  \beta_2 = \frac{29}{15},\]
coming from the values~(\ref{values-case-study}) of $\tau_0$. An explicit calculation shows that exactly four of them do: the variety defined by the equations~(\ref{redrel0})--(\ref{redrel2}) and~(\ref{for:test:hp})--(\ref{for:test:hs}) consists of four points in~$P_0$, each one of them with multiplicity one. These four points are exactly the ones in the image of~$\Psi\times\Phi^\flat$ within~$P_0$ (for these four points are in the image and satisfy all the relations).

Consider now a sufficiently small ball~$U\subset\mathbb{C}^8$ centered at~$\tau_0$. The set of common solutions to the polynomials~$h_p$, $h_q$, $h_r$ and~$h_s$ in~$\pi^{-1}(U)$ is given by~$256$ copies of~$U$. Of these, in the~$252$ where the corresponding point of~$P_0$ does not satisfy all the relations~(\ref{redrel0})--(\ref{redrel2}), no point satisfies all of these relations. The four remaining preimages of~$U$ are in~$M$, and satisfy the relations~(\ref{redrel0})--(\ref{redrel2}). They correspond to the four preimages of~$\pi$ given by theorem~\ref{thm:4points}. We have the same situation in~$\mathbb{C}^8\times F$, since for a generic~$\tau\in\mathbb{C}^8$, the four points of~$P$ in~$\pi^{-1}(\tau)$ are not in the image under~$\mu$ of the discriminant locus given by~$\Delta$.

The global picture in~$\mathbb{C}^8\times F$ is the following one. There is a Zariski open subset~$W\subset\mathbb{C}^8$ such that, in~$\pi^{-1}(W)\subset \mathbb{C}^8\times F$, $g_1,\ldots,g_5$ and~$h_p$, $h_q$, $h_r$, $h_s$ do not have poles, and where the variety~$W'$ of common solutions to these equations gives a fourfold cover~$\pi|_{W'}:W'\to W$. This cover has a holonomy group and, associated to it there is a   ramified cover~$\nu:B\to \mathbb{C}^8$. After pullback, within~$B\times F$, there are four rational maps~$s_i:B \dashrightarrow   F$, $i=1,\ldots,4$, that give sections of~$\pi$ which project onto~$W'$. The ramified cover~$\nu$ induces a field extension~$\mathcal{M}(B)/\mathcal{M}(\mathbb{C}^8)$ and thus an embedding 
$\mathcal{M}(\mathbb{C}^8)\otimes_\mathbb{C}\mathcal{O}(F)\hookrightarrow \mathcal{M}(B)\otimes_\mathbb{C}\mathcal{O}(F)$.
The variety defined by~$I$ in~$\mathcal{M}(B)\otimes_\mathbb{C}\mathcal{O}(F)$ is given by the four points~$s_i$, and does not have further points in any extension of~$\mathcal{M}(B)$. 

Let~$g\in J_M\subset \mathcal{O}( \mathbb{C}^8)\otimes_\mathbb{C}\mathcal{O}(F)$, let~$i(g)\in \mathcal{M}( \mathbb{C}^8)\otimes_\mathbb{C}\mathcal{O}(F)$ considered as an element of~$\mathcal{M}( B)\otimes_\mathbb{C}\mathcal{O}(F)$.  The variety it defines in~$B\times F$ contains the graphs of~$s_i$ for~$i=1,\ldots, 4$.\end{proof}

\begin{remark} Our result does not say that the ideal of~$M$ in~$\mathcal{O}( \mathbb{C}^8)\otimes_\mathbb{C} \mathcal{O}(F)$ is the one generated by~$g_1,\ldots,g_5$ and by the polynomials~$h_p$, $h_q$, $h_r$ and~$h_s$. It does imply that any irreducible component of the variety defined by the latter that is not contained in~$M$ projects via~$\Pi$ into a nowhere dense Zariski closed subset of~$\mathbb{C}^8$.
\end{remark}

\subsubsection{The independence of the new relations}
Let us now prove that the relation of theorem~\ref{thm:new-relation} corresponding to~$p$ is independent from relations~(\ref{knowneq0})--(\ref{rel-CS}).\footnote{The program~\texttt{indep.sage} contains the details and computations of this part.}

Fix~$\beta_0$, $\beta_1$ and~$\beta_2$. In~$\mathbb{C}[p,q,r,s]$, let~$J_\beta$ be the ideal generated by the polynomials~$R_0$, $R_1$ and~$R_2$ defined in~(\ref{redrel0})--(\ref{redrel2}) and let~$V$ be the associated variety.  Let~$W$ be the variety defined by
\[dp\wedge dR_0\wedge dR_1\wedge dR_2= 0.\]
Suppose that~$V\setminus W\neq\varnothing$ and let~$z\in V\setminus W$. Since~$z\notin W$,  $dR_0\wedge dR_1\wedge dR_2$ does not vanish at~$z$, and the function~$(R_0,R_1,R_2)$ is a submersion in a neighborhood of it. In particular, $V$ is smooth at~$z$. Furthermore, since~$dp\wedge dR_0\wedge dR_1\wedge dR_2$ does not vanish at~$z$,  $p$ is a local coordinate for~$V$ in a neighborhood of~$z$. Relations~(\ref{knowneq0})--(\ref{rel-CS}) are satisfied while~$p$ attains infinitely many values (equivalently, the image of~$V$ under~$p$ is a curve). Since the relation of theorem~\ref{thm:new-relation} corresponding to~$p$ constrains~$p$ to have at most four values, this relation is independent of the previous ones.  It is thus sufficient to prove that such~$z$ actually exists, that there exist~$\beta_0$, $\beta_1$ and~$\beta_2$ such that~$V\setminus W\neq\varnothing$, this is, such that~$V$ is not contained in~$W$. On its turn this is equivalent to the fact that the generator of the ideal of~$W$,
\[\det\left(\frac{\partial(p,R_0,R_1,R_2)}{\partial(p,q,r,s)}\right),\]
is not contained in~$J_\beta$.  This indeed happens for~$(\beta_0,\beta_1,\beta_2)=(1,0,0)$. This proves that, in~$\mathbb{C}^8\times P$, $h_p$ is independent from  relations~(\ref{redrel0})--(\ref{redrel2}) and this implies that, in~$\mathbb{C}^8\times P$, $h_p$ is independent from the known relations.

In the same way we prove that the relations of  theorem~\ref{thm:new-relation} corresponding to~$q$, $r$ and~$s$ are independent from the known ones (the same values of~$\beta_i$ allow  to conclude).

\subsection{A case study}\label{sec:case-study}

Let us consider a particular example and analyze it in detail\footnote{All the computations regarding this example can be found in the ancillary file \texttt{case-study.ipynb}.}. Let $f$ be the map defined by the rule
\begin{equation}  \label{map-case-study}
  [z_1:z_2:z_3] \longmapsto [z_1(z_1 + 4z_2 + 2z_3) : 2z_1z_2 + 3z_2^2 : 4z_1z_3 + 5z_2z_3 - z_3^2] 
\end{equation}
together with its invariant line~$\{z_1=0\}$. The fixed points of $f$ away from the invariant line are at $p_1=[1:0:0]$, $p_2=[1:1:0]$, $p_3=[1:0:1]$ and~$p_4=[5:-3:4]$. The trace and determinant of  $\mathbf{I}-Df|_{p_i}$ are:
\begin{equation} \label{values-case-study}
\begin{aligned}
  t_1 &=  -4, & d_1 &=  3,\\
  t_2 &=  -3/5, & d_2 &=  -4/25,\\
  t_3 &=  4/3, & d_3 &=  1/3,\\
  t_4 &=   9, & d_4 &=  -60.
\end{aligned}
\end{equation}
 
We have computed the rank of the derivative of the map $\Psi$ (which appears in section~\ref{sec:mapPsi}) at $f$ and verified that it is not zero. We have also used the algorithm described in section~\ref{sec:test} to confirm that there exist exactly four self-maps that realize the values $t_i,d_i$ given above, and computed their explicit expressions (which we omit here).

On the invariant line $z_1=0$, the fixed points are $p_5 = [0:1:0]$, $p_6 = [0:1:2]$, and $p_7 = [0:0:1]$. 
The eigenvalues $(u_i,v_i)$ of $\mathbf{I}-Df|_{p_i}$ (where $u_i$ is tangent to the line $z_1=0$) are given by:
\begin{align*}
u_5 &= -2/3, & v_5 &= -1/3, \\
u_6 &= 2/3, & v_6 &= -5/3, \\
u_7 &= 1, & v_7 &= 3.
\end{align*}
These define the following values for the four symmetric functions introduced in the statement of theorem~\ref{thm:new-relation}:
\[
  e_{1,0} =  1, \quad e_{3,0} =  -4/9, \quad e_{0,1} =  1, \quad e_{1,1} = -10/9.
\]

We have applied the algorithm in section~\ref{sec:computing-h} to compute the polynomials $h_\sigma$ that appear in the statement of theorem~\ref{thm:new-relation} for each of the functions $p=e_{1,0}$, $q=e_{3,0}$, $r=e_{0,1}$ and $s=e_{1,1}$. Let~$\tau_0\in\mathbb{C}^8$ be the point given by the values~$(t_1,d_1,\ldots,t_4,d_4)$ defined by~(\ref{values-case-study}). The polynomials $h_\sigma\vert_{\tau_0}$  are:
\begin{center}
\begin{flalign}\label{for:test:hp} \nonumber 
h_p\vert_{\tau_0} &= 3304458636693875651644773p^4-23379088345478790415995340p^3 &\\ \nonumber 
&+ 49551679403386908799808694p^2-39615490609470050079352860p &\\ 
&+ 10138440914868056043894733, &
\end{flalign}
\begin{flalign}   \nonumber 
h_q\vert_{\tau_0} &= 802983448716611783349679839q^4+6057061832873045850450465888q^3 &\\ \nonumber 
&+ 8220892437890168859863219744q^2+1049813648791026016967518720q &\\
&- 656868792374273004661408000, &
\end{flalign}
\begin{flalign}   \nonumber 
h_r\vert_{\tau_0} &= 3304458636693875651644773r^4+236886089402261035384796r^3 &\\ \nonumber 
&- 2361480475447712921087794r^2-1054408391045835895722692r &\\ 
&- 125455859602587870219083, &
\end{flalign}
\begin{flalign}\label{for:test:hs} \nonumber 
h_s\vert_{\tau_0} &= 802983448716611783349679839s^4-255708691033085259944247216s^3 &\\ \nonumber 
& -1889198858093497242051156664s^2-605423311222765579888256320s &\\ 
&+ 85011752062106601163535600. &
\end{flalign}
\end{center}
It is straightforward to verify that none of the above polynomials have repeated roots, which concludes the proof of theorem~\ref{thm:1point}.

\section{The test and effective computations}

In this section we describe two procedures.  The first one tests if a given collection of fourteen numbers $u_1,v_1,\ldots,u_7,v_7$ is realizable as the collection of multipliers of a self-map. It is relatively independent of our previous results and will be described in section~\ref{sec:test}.
The second procedure takes as input a rational function $\sigma(u_5,v_5,\ldots,u_7,v_7)$ invariant under the action of $S_3$ and values for $t_1,d_1,\ldots,t_4,d_4$ and produces the evaluation of polynomial~$h_\sigma$  that appears in theorem~\ref{thm:new-relation} at~$t_1,d_1,\ldots,t_4,d_4$. We will describe it in section~\ref{sec:computing-h}. Both procedures have been implemented as effective algorithms in the case when the input variables are assigned rational values.
\footnote{The code for these algorithms can be found in the programs \texttt{Test.sage} and \texttt{compute-h} (the last one available as a~\texttt{.sage} script and as an interactive~\texttt{.ipynb} notebook).}

\subsection{Theoretical formulation of the test}\label{sec:test}

Let~$F=(\mathbb{C}^2)^3/S_3$ together with its embedding into~$\mathbb{C}^9$ given by the nine elementary symmetric polynomials~$e_{10}$, $e_{20}$, $e_{30}$, $e_{01}$, $e_{11}$, $e_{21}$, $e_{01}$, $e_{02}$, $e_{03}$. Let~$\widetilde{\Psi}:E^{(4)}\to \mathbb{C}^8\times \mathbb{C}^9$ be given by
\[f\mapsto(t_1,d_1,\ldots,t_4,d_4,e_{10},\ldots,e_{03}).\]
It takes values in~$\mathbb{C}^8\times F$. 
As in section~\ref{sec:mapPsi}, the map $\widetilde{\Psi}$ factors through the quotient $E^{(4)} \sslash \operatorname{Aut}(\mathbb{P}^2,\ell)$, so we may replace $E^{(4)}$ by $E^{(4)} \sslash \operatorname{Aut}(\mathbb{P}^2,\ell)$. This can be done  by restricting to the elements of~$E^{(4)}$ that have~$z_1=0$ as invariant line and that have fixed points at~$[1:0:0]$, $[1:0:1]$ and~$[1:1:0]$. The variety of these is a Zariski-open subset of~$\mathbb{C}^8$, and   can be given parameters~$(w_1,\ldots,w_8)$. Such a parametrization allows us to consider~$\widetilde{\Psi}$ as a map~$\widetilde{\Psi}:\mathbb{C}^8\to \mathbb{C}^{8+9}$. It is rational and may be explicitly given. Let~$V\subset E^{4}\times \mathbb{C}^{17}$ be the Zariski closure of its graph.  Each coordinate~$g$ of~$\widetilde{\Psi}$ is a rational function of~$w_1,\ldots, w_8$, $g=p(w)/q(w)$ and  gives the polynomial~$q(w)g-p(q)$ in
\[\mathbb{C}[t_1,d_1,\ldots,t_4,d_4,e_{10},\ldots,e_{03},w_1,\ldots,w_8],\]
whose coefficients belong to~$\mathbb{Q}$. The ideal generated by these polynomials contains the ideal of the variety~$V$, and defines a variety~$W\subset E^{(4)}\times \mathbb{C}^{17}$ containing~$V$.
It contains the points in~$V$ which are  realizable by a self-map, and may also contain some irreducible components coming from indeterminacy points of the extension of~$\widetilde{\Psi}$ to the closure of~$E^{(4)}$ within the corresponding variety of marked self-maps.

Choosing concrete rational values for~$u_i,v_i$ determines rational values for~$t_1$, $d_1$, \ldots, $t_4$, $d_4$, $e_{10}$, \ldots, $e_{03}$.   This amounts to choosing a ``horizontal'' plane $\Pi$ in $ E^{(4)} \times \mathbb{C}^{17}$.  After evaluation on these concrete values, the polynomials defining~$W$ become polynomials in~$\mathbb{Q}[w_1,\ldots,w_8]$. The existence of points in~$W\cap \Pi$ is equivalent to the existence of common  solutions (in~$\overline{\mathbb{Q}}$) to these  polynomial equations with rational coefficients. Although they may be difficult to solve, Gr\"obner bases algorithms may be used to determine if these equations have common solutions in~$\overline{\mathbb{Q}}$. The first part of our test is to determine if~$W\cap \Pi$ is empty (in which case the data is not realizable), or not. If~$W\cap \Pi$ is not empty, further analysis is needed in order to discard points in the irreducible components of~$W$ that do not contain~$V$. 

As discussed in the proof of theorem~\ref{thm:4points}, the values for~$t_1$, $d_1$, \ldots, $t_4$, $d_4$ determine the position of the fourth fixed point that is not in~$\ell$, which is given by~$[A:B:C]$ for some explicit polynomials~$A$, $B$, $C$ in~$w_0,\ldots,w_8$.  Let us say that the position of the fourth point is \emph{nondegenerate} if~$A$ does not vanish or if~$A$ vanishes but~$B$ and~$C$ do so as well (recall that~$\ell$ is given by~$z_1=0$). If this condition is not satisfied, the fourth point belongs to~$\ell$, and corresponds to a degenerate self-map that is not in~$E^{(4)}$.

Finally, we shall say that the Test is \emph{passed} if~$W\cap\Pi$ is not empty and if the position of the fourth point is nondegenerate. Otherwise, we will say that the Test is \emph{not passed}. If the Test is not passed we are certain that there is no element of~$E^{(4)}$ realizing the data. If the Test is passed, then, in principle, further analysis is needed to check if there is a self-map in~$E^{(4)}$ realizing the corresponding data (the Test may give \emph{false positives} but no \emph{false negatives}).

\subsection{Computing the polynomial \texorpdfstring{$h_\sigma$}{h}}\label{sec:computing-h}
Let~$\sigma$ belong to the field of rational functions on~$(\mathbb{C}^2)^3$ invariant under the action of~$S_3$. In a similar way as above, we may construct the polynomial $h_\sigma$ whose existence is claimed in theorem~\ref{thm:new-relation}.

Let $\Psi\times\psi_\sigma\colon E^{(4)} \sslash \operatorname{Aut}(\mathbb{P}^2,\ell) \to \mathbb{C}^8\times\mathbb{C}$ be the map 
\[
\psi_\sigma\colon f\mapsto (t_1,\ldots,d_4, \sigma(u_5,\ldots,u_7)),
\]
and let $V_\sigma$ be its graph. Assume now that values for the variables $t_i,d_i$ are given, and let $\Pi$ be the plane in~$\left( E^{(4)} \sslash \operatorname{Aut}(\mathbb{P}^2,\ell) \right) \times \mathbb{C}^8\times\mathbb{C}$ where the variables $t_i,d_i$ take the given values. 
According to theorem~\ref{thm:4points}, for generic values, the intersection of $\Pi$ with the graph $V_\sigma$ consists of four points. We may project those four points to the last $\mathbb{C}$ factor to recover the four values that $\sigma$ may take, given that the values for~$t_i,d_i$ have been fixed. 
These four points in~$\mathbb{C}$ form an algebraic set whose vanishing ideal $\pi_\ast I_\Pi$ is generated precisely by the polynomial $h_\sigma$.

If the given values for $t_i,d_i$ are rational numbers, we may effectively compute the polynomial~$h_\sigma$ by computing a Gr\"{o}bner basis for $I_\Pi$ and eliminating all variables coming from $E^{(4)}$ (which geometrically corresponds to taking the projection $\pi$ onto the $\sigma$-axis).

\section{An application to the Painlev\'e analysis of some quadratic homogeneous equations}\label{sec:pain}

In this section, as an application of the previously obtained results, we will study some ordinary differential equations in the complex domain and try to determine those, within a particular class, which do not have multivalued solutions.

There is a natural correspondence between quadratic homogeneous vector fields on~$\mathbb{C}^3$ and quadratic self-maps of~$\mathbb{P}^2$. Many relevant notions for a vector field are also relevant for the associated self-map and vice versa. To a quadratic homogeneous vector field~$X=\sum_i P_i\indel{z_i}$ on~$\mathbb{C}^3$ we associate the self-map~$f:\mathbb{P}^2\to \mathbb{P}^2$
\[[z_1:z_2:z_3]\mapsto [P_1(z_1,z_2,z_3):P_2(z_1,z_2,z_3):P_3(z_1,z_2,z_3)],\] 
and from the self-map we can recover the vector field up to a constant multiple. This correspondence is equivariant with respect to the natural action of~$\mathrm{GL}(3,\mathbb{C})$ on both spaces. 

Under this correspondence, invariant lines for~$f$ give hyperplanes which are tangent to~$X$; indeterminacy points of~$f$ correspond to lines through the origin of~$\mathbb{C}^3$ along which~$X$ vanishes identically; fixed points of~$f$ correspond to \emph{nondegenerate radial orbits} of~$X$, lines through the origin of~$\mathbb{C}^3$ along which~$X$ does not vanish identically and along which it is collinear with the radial vector field~$\sum_iz_i\indel{z_i}$.  For the nondegenerate radial orbit of~$X$ corresponding to the fixed point~$p$ of~$f$, its \emph{Kowalevski exponents} are the eigenvalues of~$\mathbf{I}-Df|_p$ (see~\cite{goriely} for a direct definition). The Kowalevski exponents of a nondegenerate radial orbit give a first-order approximation of the behavior of the vector field along it.

Relations~(\ref{knowneq0})--(\ref{rel-CS}) may be interpreted within the context of vector fields and foliations. A quadratic homogeneous vector field~$X$ on~$\mathbb{C}^3$ induces naturally a foliation~$\mathcal{F}_X$  on~$\mathbb{P}^2$. Relation~(\ref{knowneq2}) follows from Baum and Bott's index theorem~\cite{Baum-Bott} for~$\mathcal{F}_X$. If~$f$ has an invariant line~$\ell$, this line is also tangent to~$\mathcal{F}_X$, and relation~(\ref{rel-CS}) can be directly obtained  from the Camacho-Sad index theorem~\cite[appendix]{Camacho-Sad}. A quadratic homogeneous vector field on~$\mathbb{C}^3$ also induces a foliation on~$\mathbb{C}^3$.  Relations~(\ref{knowneq1})--(\ref{knowneq2}) can  be obtained in a straightforward manner from  index theorems applied to   foliations related to this one (see the final comments in~\cite{guillot-pointfixe}). As far as algebraic relations go, they can be established in one context and transported to the other, but a general conceptual link between the two  is still missing (see, however, \cite{obrian}, for a direct relation between Lefschetz's  and Baum and Bott's theorems).
 
The previously established results, like  Theorem~\ref{thm:new-relation},   can be interpreted as results about the Kowalevski exponents of the quadratic homogeneous vector fields which have an invariant plane.

For ordinary differential equations in the complex domain, we may have multivalued solutions. For instance, for~$\alpha\in\mathbb{C}$, the vector field~$-x^2\indel{x}+\alpha xy\indel{y}$, corresponding to the system of equations~$x'=-x^2$, $y'=\alpha xy$,  has the solution~$(x(t),y(t))=(1/t,t^{\alpha})$, which is multivalued if~$\alpha\notin\mathbb{Z}$. Since the middle of the XIXth century, the problem of understanding the ordinary differential equations in the complex domain which do not have multivalued solutions has been a central one, with Painlev\'e's call upon a systematic and exhaustive approach~\cite{painleve} leading to many investigations. While there are many settings where this program has been successfully carried out, in other ones, like that of autonomous equations given by algebraic vector fields on affine threefolds, or even polynomial vector fields on~$\mathbb{C}^3$, we know very little. Even for the equations given by quadratic homogeneous vector fields on~$\mathbb{C}^3$, we do not have  an understanding of the ways in which such a vector field may be free of multivalued solutions  (the study of homogeneous vector fields is a key step towards  that of more general polynomial ones).

In what follows, we will use Rebelo's notion of \emph{semicompleteness}~\cite[definition~1]{guillot-rebelo} to be precise when we say that a vector field on a manifold has no multivalued solutions  when considered as an autonomous ordinary differential equation.

\begin{problem}\label{mainproblem} Classify the quadratic homogeneous vector fields on~$\mathbb{C}^3$ which are semicomplete.
\end{problem}
 
The Kowalevski exponents give an obstruction of arithmetical nature for semicompleteness: \emph{if a quadratic homogeneous vector field on~$\mathbb{C}^3$ is semicomplete and~$0$ is its only singularity, it has seven different nondegenerate radial orbits and each one of  them has two nonzero integers as Kowalevski exponents} (see~\cite[section~2.1]{guillot-semi} for a proof). Thus, in order to solve the above problem, one may start by solving the following one.

\begin{problem}  Determine the sets of seven pairs of nonzero integers that are realizable as sets of Kowalevski exponents of quadratic homogeneous vector fields on~$\mathbb{C}^3$.
\end{problem}

(The next step would be, for each vector field that realizes such a set of exponents, to investigate its semicompleteness.) The problem asks us to find integer points in an algebraic variety which, as we previously discussed, we know only partially---we only know the bigger variety defined by  equations~(\ref{knowneq0})--(\ref{knowneq2}).

The analogue problem in~$\mathbb{C}^2$ can be very satisfactorily solved. A semicomplete quadratic homogeneous vector field on~$\mathbb{C}^2$ with an isolated singularity at~$0$ has three radial orbits, each one of them with an exponent~$u_i$; these exponents are integers and are subject to relation~(\ref{rel-invline1}). The solutions to the latter are
\begin{equation}\label{numbers2d}\{(1,-m,m), (2,3,6), (2,4,4), (3,3,3)\}.\end{equation}
For each one of these there is a unique (up to linear equivalence) quadratic homogeneous vector field which is, moreover, semicomplete.

In~$\mathbb{C}^3$, despite the fact that we do not have all the relations that bind the Kowalevski exponents, we may still go a long way in the solution of problem~\ref{mainproblem}.

For~$d_i=u_iv_i$, we may write~(\ref{knowneq0}) as
\begin{equation}\label{R0}\sum_{i=1}^7 \frac{1}{d_i}=1.\end{equation}
An integer solution to this equation gives rise to finitely many integer solutions of the system~(\ref{knowneq0})--(\ref{knowneq2}): for each~$d_i$ there are finitely many couples of integers~$(u_i,v_i)$ such that~$u_iv_i=d_i$ and each one of these sets may tested for compliance with  equations~(\ref{knowneq1})--(\ref{knowneq2}). In this sense, equation~(\ref{R0}) is the master one. A solution $(d_1,\ldots,d_7)$ to equation~(\ref{R0}) is said to \emph{belong to the $n$\textsuperscript{th} family} if~$n$ is the smallest cardinality of the subsets~$J\subset\{1,\ldots,7\}$ such that~$\sum_{i\in J}1/d_i=1$.

A classification of the semicomplete quadratic homogeneous  vector fields on~$\mathbb{C}^3$ having an isolated singularity at the origin associated to the first or second families of the solutions of~(\ref{R0}) appears in~\cite{guillot-qc3} (there are infinitely many of them). The strategy for this classification is a mixed one: from a partial knowledge of the solutions to the Diophantine system~(\ref{knowneq0})--(\ref{knowneq2}) we may look for other obstructions for semicompleteness in the quadratic vector fields having the associated data as exponents; these entail other relations among the Kowalevski exponents and so on. The possibility of obtaining an analogue classification for some of the other families relies on the possibility of describing the solutions to system of Diophantine equations together with the geometric constrains imposed on the differential equations that come with a solution or a family of solutions. For the seventh family, however, there seems to be no good starting point other than enumerating the solutions of~(\ref{R0}). We may use the previously obtained results to shed some light on the semicompleteness of the equations of the seventh family which have an invariant plane. 

\begin{fact}\label{mainfact} The sets of Kowalevski exponents of the semicomplete quadratic homogeneous vector fields on~$\mathbb{C}^3$ which have an isolated singularity at the origin, which belong to the seventh family, and which have an invariant plane, belong to the list of table~\ref{tab:admissinvplane}. The data in this list belong to the variety of admissible spectra.\end{fact}
 
\begin{table}
\begin{tabular}{r|l|l}
\# & \begin{tabular}{l} Kowalevski exponents of the \\  radial orbits away from~$\ell$ \end{tabular} & \begin{tabular}{l} Kowalevski exponents of \\ the radial orbits within~$\ell$ \end{tabular}  \\ \hline
1 &  $(2,5)$, $(1,3)$, $(1,3)$,    $(-5,12)$ & $(1,2)$, $(5,-1)$, $(-5,4)$ \\
2 & $(2,5)$, $(1,3)$, 
$(1,3)$, $(-1,4)$ & $(1,2)$, $(17,-5)$, $(-17,12)$   \\
3 & $(1,3)$, $(2,5)$, $(-5,17)$, $(12,-17)$  & $(1,2)$, $(1,3)$, $(-1,4)$  \\
4 & $(1,3)$,  $(2,5)$, $(3,4)$,  $(12,-17)$ & $(1,2)$, $(22,-5)$, $(-22,17)$  \\
5 & $(1,3)$, $(2,5)$, $(3,4)$, $(-5,17)$ & $(1,2)$, $(29,-17)$, $(-29,12)$ \\
6 & $(1,3)$, $(2,5)$, $(2,5)$, $(10,-13)$ & $(1,2)$, $(16,-3)$, $(-16,13)$  \\
7 & $(1,3)$, $(2,5)$, $(2,5)$, $(-3,13)$ & $(1,2)$, $(23,-13)$, $(-23,10)$  \\
8 &  $(1,4)$, $(2,3)$, $(2,3)$, $(-3,8)$ & $(1,2)$, $(11,-3)$, $(-11,8)$  \\
9 &  $(1,4)$, $(2,3)$, $(2,3)$, $(-7,12)$ & $(1,2)$, $(9,-2)$, $(-9,7)$  \\
10  &  $(1,4)$, $(2,3)$, $(2,3)$, $(-2,7)$ & $(1,2)$, $(19,-7)$, $(-19,12)$  
\end{tabular}
\caption{The ten realizable sets of exponents of fact~\ref{mainfact}.} \label{tab:admissinvplane}
\end{table}

Let us establish fact~\ref{mainfact} by using the kind of mixed strategy previously mentioned.

There are finitely many solutions to equation~(\ref{R0}) belonging to the seventh family. In fact, we may order~$d_1,\ldots, d_7$ so that
\begin{itemize}
\item $d_1>0$;
\item if~$d_i$ and~$d_j$ are positive and $i<j$, $d_i<d_j$; if~$d_i$ and~$d_j$ are negative and $i<j$, $d_i>d_j$;
\item $d_{j+1}>0$ if~$\sum_{i=1}^j 1/d_i < 1$;  $d_{j+1}<0$ if~$\sum_{i=1}^j 1/d_i > 1$.
\end{itemize}
With such an order, $d_1$ must belong to~$\{1,\ldots,7\}$, since \[1=\sum_{i=1}^7 \frac{1}{d_i} \leq \sum_{d_i>0} \frac{1}{d_i}\leq \sum_{d_i>0} \frac{1}{d_1}\leq \frac{7}{d_1}.\] There are thus finitely many choices for~$d_1$. Given $d_1,\ldots, d_j$, there are, by the same reasons, finitely many choices for~$d_{j+1}$. These observations lead  to an algorithm to list all the solutions to~(\ref{R0}) belonging to the seventh family. From it, we may construct all the solutions to the system~(\ref{knowneq0})--(\ref{knowneq2}) belonging to the seventh family.
 
If~$d_1=2$, either~$(u_1,v_1)=(  1,  2)$ or~$(u_1,v_1)=(-1,-2)$. In both cases, the semicompleteness of the vector field implies that there is an invariant plane  tangent to the exponent~$\pm 1$ (see lemma~7 in~\cite{guillot-qc3}). On its turn, this implies, to begin with, that relations~(\ref{rel-invline1})--(\ref{rel-CS}) must hold among the exponents of the radial orbits that belong to this invariant plane.

If~$(u_1,v_1)=(1,2)$, for the other two radial orbits contained in the invariant plane, the exponents tangent to it are, from~(\ref{numbers2d}), necessarily of the form~$m$ and~$-m$. Thus, from~(\ref{rel-invline1})--(\ref{rel-CS}), the exponents of these two radial orbits are of the form~$(n,m)$ and~$(n+m,-m)$. Equation~(\ref{R0}) becomes
\[\frac{1}{2}+\sum_{i=2}^5\frac{1}{d_i}+\frac{1}{n(n+m)}=1,\]
and we may start by finding the integer solutions of the equation
\begin{equation}\label{cincounmedio} \sum_{i=1}^5 \frac{1}{k_i}=\frac{1}{2}\end{equation}
such that the sum of no proper subset of the summands is~$1/2$ (if~$d_i=2$ or~$n(n+m)=2$ we would be either in the first or in the second family). From one solution to this equation we may find all the possible values of~$m$ and~$n$ and those of $d_2$, \ldots, $d_5$;  from the latter we may obtain all the values of~$u_2$, $v_2$, \ldots, $u_5$, $v_5$ such that~$u_iv_i=d_i$. After eliminating the solutions that do not belong to the seventh family, we may investigate for compliance with the relations~(\ref{knowneq0})--(\ref{knowneq2}). All this is carried out by the program in the ancillary file~\texttt{7th\_2pos.sage}.  As output we obtain the list in the file~\texttt{7th\_2pos.out}. It is a list of all the integer solutions to equations~(\ref{knowneq0})--(\ref{knowneq2}) with~$(u_1,v_1)=(1,2)$ and with an invariant plane in the previously explained configuration. We may subject all the items in this list to our~Test (cf.~section~\ref{sec:test}).\footnote{This is done by the program \texttt{testlist.sage}. }
Those that pass the test appear in table~\ref{tab:admissinvplane}.\footnote{They correspond, respectively, to lines~1, 2, 90, 91, 92,  88, 89, 101, 103 and~104 in the file \texttt{7th\_2pos.out}; line 102 passes the test but is, up to ordering, identical to line~101.} Quite remarkably, every  set in the table may be realized by a self-map: even if the Test may, in principle, give false positives,  none is to be found in the present sample.

If~$(u_1,v_1)=(-1,-2)$, by~(\ref{numbers2d}), for the other two radial orbits contained in the invariant plane, the exponents tangent to the plane are both equal to~$1$. The exponents of these two radial orbits are of the form~$(1,m)$ and~$(1,1-m)$. Equation~(\ref{R0}) becomes
\[\frac{1}{2}+\sum_{i=2}^5\frac{1}{d_i}+\frac{1}{m(m+1)}=1-\frac{1}{2}.\]
It is also subordinate to equation~(\ref{cincounmedio}). We may proceed as before. The program is in the file~\texttt{7th\_2neg.sage}. Quite surprisingly, there are no solutions at all! There is no need in this case to apply our test.

Lastly, suppose that~$d_1\geq 3$. We can determine the sets of seven pairs of nonzero integers  that belong to the seventh family and for which~$d_i\geq 3$. This is done by the program \texttt{7th\_geq3.sage}. There are around 160 such sets (there are some repetitions); they may be found in the file~\texttt{7th\_geq3.out}, which we will not reproduce here. From each unordered set of seven unordered pairs appearing in this list,  we may look for (unordered) triples of ordered pairs of    exponents satisfying relations~(\ref{rel-invline1})--(\ref{rel-CS}). These are those which have, potentially, an invariant plane. This is done by the program~\texttt{extract\_lines.sage}. After eliminating repetitions, we obtain only the three admissible sets:
\begin{itemize}
\item  $(1,3)$, $(2,2)$, $(-3,28)$, $(-10,-14)$, $(3,1)$, $(1,10)$, $(-3,28)$,
\item   $(-2,-2)$, $(2,2)$, $(1,5)$, $(1,-15)$, $(1,3)$, $(14,2)$, $(-14,30)$,
\item   $(1,3)$, $(1,5)$, $(2,3)$, $(-4,45)$, $(1,4)$, $(-5,-3)$, $(5,-18)$.
\end{itemize}
They all belong to different solutions of~(\ref{R0}).
We may  submit these three to our Test, which they do no not pass, and this certifies that they are not realizable. This completes the argument establishing fact~\ref{mainfact}.

\subsection{Integrating the equations} We have established that the sets of Kowalevski exponents of the semicomplete vector fields in the class under consideration belong to table~\ref{tab:admissinvplane}, and a natural question is if the vector fields whose exponents belong to this list are semicomplete. We would like to finish by discussing the integration of some of these equations. However, since this integration is not the central point of this article and since the complete details of this integration would certainly require a more extended presentation, we will give only a brief account of our findings. They prove that some of the vector fields whose exponents belong to  table~\ref{tab:admissinvplane} are semicomplete and strongly suggest that the other ones are semicomplete as well.\\

The vector fields with  an invariant plane whose exponents are in table~\ref{tab:admissinvplane} have all a homogeneous polynomial first integral and thus, whenever the vector fields are semicomplete, they do so in very specific and understood ways (see Theorem~B in~\cite{guillot-rebelo}). We succeeded in integrating some of these equations (for instance, we established that all the vector fields associated to the first through the seventh sets of exponents in table~\ref{tab:admissinvplane} are semicomplete). Like in~\cite{guillot-qc3}, we witness the ubiquity of vector fields having homogeneous first integrals whose generic level set identifies to a Zariski open set of an Abelian surface where the vector field becomes a linear one.

For each one of the first seven sets of exponents of table~\ref{tab:admissinvplane}, there are two vector fields realizing the data. Their coefficients belong to a quadratic extension of~$\mathbb{Q}$ and are Galois conjugates of one another.  
 
All the equations coming from data~1 through~5 are defined  over~$\mathbb{Q}(\sqrt{3})$ and have homogeneous polynomial first integrals of degree~$12$.  They are integrated by the solutions of the homogeneous Chazy~X equation
\begin{equation}\label{chazyx}\phi'''   =  6\phi^2\phi'+3\frac{9+7\sqrt{3}}{11}(\phi'+\phi^2)^2
\end{equation}
and their derivatives. This equation, together with the Chazy~IX equation that will occur shortly, appear in Chazy's study of third-order equations~\cite{chazy}. When seen as vector fields on~$\mathbb{C}^3$, they have first integrals and, in restriction to a generic level set, they identify to linear flows on Zariski-open subsets of Abelian surfaces; in consequence, they are semicomplete (see~\cite{guillot-chazy} for a detailed account). For instance, the vector field for data~3 reads
\begin{multline*} 
z_1[5(1+\sqrt{3})z_1+80(1+\sqrt{3})z_2-(107-12\sqrt{3})z_3]\frac{\partial}{\partial z_1}+\\
+[90(1+\sqrt{3})z_1z_2+(757-433\sqrt{3})z_1z_3-(1+5\sqrt{3})z_2^2+\\ 42(11-6\sqrt{3})z_3z_2-(757-433\sqrt{3})z_3^2]\frac{\partial}{\partial z_2}+\\
-z_3[55(1+\sqrt{3})z_1+115(1+\sqrt{3})z_2+(47-72\sqrt{3})z_3]\frac{\partial}{\partial z_3},
\end{multline*}
and the linear function 
\[\phi= 5(5+\sqrt{3})(z_2-z_1)+(56\sqrt{3}-94)z_3\]
is a solution to~(\ref{chazyx}). From this expression and its first and second derivatives, one can obtain rational expressions for~$z_1$, $z_2$ and~$z_3$ in terms of~$\phi$, $\phi'$ and~$\phi''$ and, through this, we establish that the vector field is semicomplete. We can similarly integrate equations~1, 2, 4 and~5, which are also semicomplete.

In a similar manner, the equations coming from data~6 and~7, defined over~$\mathbb{Q}(\sqrt{5})$, have homogeneous polynomial first integrals of degree~$10$ and are integrated by the solutions of the homogeneous Chazy~IX equation
\begin{equation}\label{chazyix} \phi'''  = 18(\phi'+\phi^2)(\phi'+3\phi^2)-6(\phi')^2
\end{equation}
For instance, the vector field for data~7 reads 
\begin{multline*}  
3z_1[2\sqrt{5}z_1+(6-5\sqrt{5})z_2-(6+5\sqrt{5})z_3]\frac{\partial}{\partial z_1}+\\
+[(67+30\sqrt{5})(z_1-z_3)z_3+(13-6\sqrt{5})z_1z_2+(5-3\sqrt{5})z_2^2+(82+48\sqrt{5})z_3z_2]\frac{\partial}{\partial z_2}+\\
-[(67-30\sqrt{5})(z_1-z_2)z_2+(13+6\sqrt{5})z_1z_3+(5+3\sqrt{5})z_3^2+(82-48\sqrt{5})z_3z_2]\frac{\partial}{\partial z_3},
\end{multline*}
and the linear function
\[\phi  =\frac{1}{2}(9-3\sqrt{5})z_2-\frac{1}{2}(9+3\sqrt{5})z_3\]
is a solution to~(\ref{chazyix}). From it we can obtain rational expressions for~$z_1$, $z_2$ and~$z_3$ in terms of~$\phi$, $\phi'$ and~$\phi''$, establishing that the vector field is semicomplete. We have an analogue phenomenon for the vector fields of data~6, which are also semicomplete.

The  equations coming from data~8 in table~\ref{tab:admissinvplane} are all defined over~$\mathbb{Q}(\sqrt{2})$; there are six of them (three pairs of Galois conjugates). They all have a homogeneous polynomial first integral of degree~8 and commute with a  homogeneous vector field of degree~$12$, which becomes a rational one of degree~$4$ after dividing by the first integral. The vector field
\begin{multline}\label{case8}z_1[3z_1+(9\sqrt{2}-15)z_2-4z_3]\frac{\partial}{\partial z_1}+
z_2[10z_3-9z_1+(9\sqrt{2}-3)z_2]\frac{\partial}{\partial z_2}+ \\ +
[36(1-\sqrt{2})z_2z_3+18(5-3\sqrt{2})(z_1-z_2)z_2-z_3^2]\frac{\partial}{\partial z_3}\end{multline}
gives two of them (one for each determination of~$\sqrt{2}$). Setting
\[y=z_3^2+9(3-2\sqrt{2})(z_2-z_1)z_2-6(1-\sqrt{2})z_2z_3\]
for any solution to~(\ref{case8}) gives a solution to the third-order equation
\begin{equation}\label{burn}(y''')^2-24y'''y'y-192y^5+80y''y^3+120(y')^2y^2+4y''(y')^2-8(y'')^2y=0,\end{equation}
a particular case of Cosgrove's reduced F-V equation~\cite[section~5]{cosgrove-p2}
\[y^{\mathrm{(iv)}}=20yy''+10(y')^2-40y^3.\] 
This does not prove that the vector field~(\ref{case8}) is semicomplete, but  points in that direction (it proves that the solutions have finitely many determinations). Equation~(\ref{burn}) is related to the Jacobian of Burnside's curve~$\zeta^2=\xi(\xi^4-1)=0$ (see~\cite[section~5.1]{guillot-qc3} for details). All the equations coming from data~8 seem likely to be integrated by the solutions of~(\ref{burn}) and its derivatives.

All the vector fields coming from data~9 and~10 in table~\ref{tab:admissinvplane} are defined over~$\mathbb{Q}(\sqrt{3})$. They all have a homogeneous polynomial  first integral of degree~$12$ and commute with a  homogeneous polynomial vector field of degree~$8$ (those from data~9) or~$20$ (those from data~10)---in the last case the commuting vector field becomes a rational one of degree~$8$ after dividing by the first integral---.  The following conjectural geometric description would entail their semicompleteness. For~$\rho$ a primitive third root of unity and~$E_\rho=\mathbb{C}/\langle 1,\rho\rangle$, $E_\rho\times E_\rho$ has the automorphism~$\mu$ induced by~$(z,w)\mapsto(\rho^2 w, -\rho^2 z)$, and, for~$\lambda=-i\rho$,  vector fields~$X$ and~$Y$ such that~$\mu_*X=\lambda X$ and~$\mu_*Y=\lambda^7Y$~\cite[section~5.2]{guillot-qc3}.  It is likely that, for all the vector fields coming from data~9 and 10, a generic level set of the first integral identifies to a Zariski-open set of~$E_\rho\times E_\rho$ where the vector field becomes~$X$ and the commuting one~$Y$ (see also section~5.2 in~\cite{guillot-qc3}).

\appendix

\section{A relative fixed-point formula}\label{appendix:relative-formula}

We give here an elementary proof of relation~(\ref{rel-CS}) following the lines of the proof of theorem~12.4 in~\cite[section~12]{milnor-one}. A more conceptual one, based on the Atiyah-Bott fixed-point theorem, appears in the unpublished notes~\cite{ramirez2016woods}.

\begin{theorem} 
Let~$f:\mathbb{P}^2\to\mathbb{P}^2$ be a rational map, $\ell\subset\mathbb{P}^2$ a line such that~$f(\ell)\subset \ell$. Suppose that~$f|_\ell$ has the fixed points~$p_1$, \ldots, $p_{k}$ and that these are simple. Let~$\lambda_i$ and~$\mu_i$ be  be the eigenvalues of~$Df$ at~$p_i$, with~$\lambda_i$ the eigenvalue tangent to~$\ell$.  Then, 
\[\sum_{i=1}^{k} \frac{1-\mu_i}{1-\lambda_i}=1.\]
\end{theorem}
\begin{proof}In homogeneous coordinates~$[z_1:z_2:z_3]$, let~$\ell$ be given by~$z_1=0$ and~$f$ be given by
\[[z_1:z_2:z_3]\mapsto[z_1\widehat{P}(z_1,z_2,z_3):\widehat{Q}(z_1,z_2,z_3):\widehat{R}(z_1,z_2,z_3)],\]
for~$\widehat{Q}$ and~$\widehat{R}$  homogeneous polynomials of degree~$d$ and~$\widehat{P}$ a homogeneous polynomial of degree~$d-1$. Consider the affine chart~$z_3\neq 0$ and suppose, without loss of generality, that it contains the~$d+1$ fixed points of~$f|_\ell$. In the coordinates~$[x:y:1]$, $f$ is given by
\[(x,y)\mapsto\left(\frac{xP(x,y)}{R(x,y)},\frac{Q(x,y)}{R(x,y)}\right),\]
where~$P(x,y)=\widehat{P}(x,y,1)$ and so on. Let~$y_1$, \ldots, $y_{d+1}$ be the fixed points of~$f$ within~$\ell$. Consider the rational one-form in~$\ell$
\[\eta=\frac{1-\frac{P}{R}(0,y)}{y-\frac{Q}{R}(0,y)}dy.\] 
Since all the fixed points of~$f|_\ell$ are within the chosen affine chart, $R(0,y)=cy^d+\cdots$ for some~$c\neq 0$. This implies that~$\lim_{y\to\infty} P(0,y)/R(0,y)=0$ and that~$\lim_{y\to\infty} Q(0,y)/R(0,y)\in\mathbb{C}$. Hence, $\eta$ has a pole at~$\infty$ with residue~$-1$.  The other poles of~$\eta$ are given by the fixed points of~$f|_\ell$. Let us calculate the corresponding residues. On the one hand, $y-Q(0,y)/R(0,y)$, which gives the restriction of~$f$ to~$\ell$, has, for its Taylor development at~$y_i$,
\[y_i+(y-y_i)-(y_i+f|_\ell'(y_i)(y-y_i)+\cdots)=(1-\lambda_i)(y-y_i)+\cdots.\]
On the other, a straightforward computation shows that
\[\mu_i=\left.\frac{\partial}{\partial x} \left(x\frac{P}{R}\right)\right|_{(0,y_i)}=\frac{P(0,y_i)}{R(0,y_i)}.\] 
The Residue Theorem allows us to conclude.\end{proof}

\bibliography{multipliers}
\bibliographystyle{aomalpha}

\end{document}